\documentclass[10pt, leqno]{article}

\usepackage{mathrsfs}
\usepackage{amsmath}
\usepackage{mathtools}
\usepackage{amssymb}
\usepackage{amsthm}
\usepackage{comment}
\usepackage[dvipdfmx]{hyperref}

\usepackage{accents}

\newtheorem{them}{Theorem}[section]
\newtheorem{pro}{Proposition}[section]
\newtheorem{lemma}{Lemma}[section]
\newtheorem{cor}{Corollary}[section]

\numberwithin{equation}{section}

\begin{document}

\title{
A note on the Cauchy problem for $-D_0^2+2x_1D_0D_2+D_1^2+x_1^3D_2^2+\sum_{j=0}^2b_jD_j$}

\author{Tatsuo Nishitani\footnote{Department of Mathematics, Osaka University:  
nishitani@math.sci.osaka-u.ac.jp
}}

\date{}
\maketitle

\def\ep{\epsilon}
\def\dif{\partial}
\def\al{\alpha}
\def\be{\beta}
\def\ga{\gamma}
\def\om{\omega}
\def\lam{\lambda}
\def\Lam{\Lambda}
\def\de{\delta}
\def\tila{{\tilde \lambda}}
\def\varep{\varepsilon}
\def\R{{\mathbb R}}
\def\N{{\mathbb N}}
\def\C{{\mathbb C}}
\def\Q{{\mathbb Q}}
\def\Ga{\Gamma}
\def\La{\Lambda}
\def\lr#1{\langle{#1}\rangle_{\gamma }}
\def\mD{\lr{ D}_{\mu}}
\def\xim{\lr{\xi}_{\mu}}
\def\co{{\mathcal C}}
\def\op#1{{\rm op}({#1})}
\def\olr#1{\langle{#1}\rangle}
\def\bg{{\bar g}}

\begin{abstract}
In this note, we improve a previously proven non-solvability result
of the Cauchy problem in the Gevrey class for a homogeneous second-order differential operator mentioned in the title.  We prove that the Cauchy problem for this operator is not locally solvable at the origin for any lower order term in the Gevrey class of order greater than $5$, lowering the previously obtained Gevrey order $6$. 
\end{abstract}

\section{Introduction}	

In \cite{Ni:Siena} considering a second order operator in $\R^{1+2}$
 \begin{equation}
 \label{eq:model}
 P_{mod}=-D_0^2+2x_1D_0D_2+D_1^2+x_1^3D_2^2,\quad x=(x_0, x')=(x_0, x_1, x_2)
 \end{equation}
 we have proved
\begin{them}
\label{thm:main}{\rm(\cite[Theorem 1.1]{Ni:Siena})}
The Cauchy problem for $P_{mod}+\sum_{j=0}^2b_jD_j$ is not locally solvable at the origin in the Gevrey class of order $s$ for any $b_0, b_1, b_2\in \C$ if $s>6$. In particular the Cauchy problem for $P_{mod}$ is $C^{\infty}$ ill-posed near the origin for any $b_0, b_1, b_2\in \C$.
\end{them}
Recall that the Gevrey class of order $s$, denoted by $\gamma^{(s)}(\R^n)$, is the set of all $f(x)\in C^{\infty}(\R^n)$ such that for any compact set $K\subset \R^n$, there exist $C>0$, $h>0$ such that
\begin{equation}
\label{eq:gevrey}
|\dif_x^{\al}f(x)|\leq Ch^{|\al|}|\al|!^s,\quad x\in K,\;\;\al\in \N^{n}.
\end{equation}
We say that the Cauchy problem for $P=P_{mod}+\sum_{j=0}^2b_jD_j$ is locally solvable in $\gamma^{(s)}$ at the origin if for any $\Phi=(u_0, u_1)\in (\gamma^{(s)}(\R^2))^2$ there exists  a neighborhood $U_{\Phi}$, may depend on $\Phi$,  of the origin such that the Cauchy problem 
\begin{equation}
\label{eq:CP}
\begin{cases}
Pu=0\quad \text{in}\;\;U_{\Phi},\\
D_0^ju(0, x')=u_j(x'),\quad x'\in U_{\Phi}\cap\{x_0=0\},\;\;j=0, 1
\end{cases}
\end{equation}
has a solution $u(x)\in C^2(U_{\Phi})$.
In this note we remark that one can improve Theorem \ref{thm:main} so that
\begin{them}
\label{thm:main:bis}
The Cauchy problem for $P_{mod}+\sum_{j=0}^2b_jD_j$ is not locally solvable in $\gamma^{(s)}$ at the origin  for any $b_0, b_1, b_2\in \C$ if $s>5$. 
\end{them}
$P_{mod}$ plays a special role in the well-posedness of the Cauchy problem for differential operators with non-effectively hyperbolic characteristics. Before explaining it we give a short introduction to the general context.  Let $P$ be a differential operator of order $m$ with principal symbol $p(x, \xi)$.  At a singular point $\rho$ of $p=0$, the Hamilton map $F_p(\rho)$ is defined as the linearization at $\rho$ of the Hamilton vector field $H_p$. The first fundamental result is
\begin{them}
\label{thm:IPH}{\rm(\cite{IP}, \cite{Ho1})} Let $\rho$ be a singular point of $p=0$  and assume that $F_p(\rho)$ has no non-zero real  eigenvalues. If the Cauchy problem for $P$ is $C^{\infty}$ well posed it is necessary that
\begin{equation}
\label{eq:IPH}
{\mathsf{Im}}P_{sub}(\rho)=0,\quad \big|P_{sub}(\rho)\big|\leq {\rm Tr^+}F_p(\rho)/2
\end{equation}
where ${\rm Tr^+}F_p(\rho)=\sum|\mu_j|$ and $\mu_j$ are the eigenvalues of $F_p(\rho)$, counted with multiplicity, and $P_{sub}$ is the subprincipal symbol of $P$. 
\end{them}
We call \eqref{eq:IPH} the Ivrii-Petkov-H\"ormander condition (IPH condition, for short). If the strict inequality holds in \eqref{eq:IPH} we call it the strict Ivrii-Petkov-H\"ormander condition (strict IPH condition, for short). For the sufficiency of the IPH condition, we assume that the set of singular points $\Sigma$ of $p=0$ is a $C^{\infty}$ manifold and the following conditions are satisfied:
\begin{equation}
\label{eq:jyoken}
\begin{cases}
\text{$F_p$ has no non-zero real eigenvalues at each point of $\Sigma$},\\
\text{Near  each point of $\Sigma$, $p$ vanishes exactly of order $2$} \\
\text{and the rank of  $\sum_{j=0}
^nd\xi_j\wedge dx_j$ is constant. }
\end{cases}
\end{equation}
According to the spectral type of $F_p(\rho)$, two different possible cases may arise
\begin{align}
\label{eq:type:1}
{\rm Ker}\,F_p^2(\rho)\cap {\rm Im}\,F_p^2(\rho)= \{0\}, \\
\label{eq:type:2}
{\rm Ker}\,F_p^2(\rho)\cap {\rm Im}\,F_p^2(\rho)\neq \{0\}.
\end{align}
We say that $p$ is of spectral type $1$ (resp. type $2$) near $\rho\in\Sigma$ if there is a conic neighborhood $V$ of $\rho$ such that \eqref{eq:type:1} (resp. \eqref{eq:type:2}) holds in $V\cap \Sigma$. We say that there is no transition (of spectral type) if for any $\rho\in \Sigma$ one can find a conic neighborhood $V$ of $\rho$ such that either \eqref{eq:type:1} or \eqref{eq:type:2} holds in $V\cap\Sigma$. 
\begin{them}
\label{thm:jyubun} {\rm(\cite[Theorem 5.1]{Ni:book})} Assume \eqref{eq:jyoken} and that there is no transition and no bicharacteristic tangent to $\Sigma$. Then the Cauchy problem for $P$ is $C^{\infty}$ well posed  under the strict IPH condition.
\end{them}
The principal symbol  $p=-\xi_0^2+2x_1\xi_0\xi_2+\xi_1^2+x_1^3\xi_2^2$ of $P_{mod}$  is a typical example of spectral type $2$ with tangent bicharacteristics (note that there is no tangent bicharacteristic if $p$ is of spectral type $1$, see \cite{Iv2}). The set of singular point of $p=0$ with $\xi\neq 0$ is $\Sigma=\{x_1=0, \xi_0=\xi_1=0\}$ which is a $C^{\infty}$ manifold on which $p$ vanishes exactly of order $2$ and a tangent bicharacteristic is given explicitly  by
\begin{equation}
\label{eq:bitokusei}
(x_1,  x_2)=(-x_0^2/4,  x_0^5/80),\quad 
(\xi_0, \xi_1, \xi_2)=(0, x_0^3/8, 1)
\end{equation}
which is  parametrized by $x_0$. The operator $P_{mod}$ shows how the situation becomes to be complicated when a tangent bicharacteristic exists. We now give some  such results: The Cauchy problem for $P_{mod}$  is not locally solvable in $\gamma^{(s)}$ for $s>5$, in particular ill-posed in $C^{\infty}$ while the Cauchy problem for general second order operator $P$ of spectral type $2$ satisfying $P_{sub}=0$ on $\Sigma$ (note that the subprincipal symbol of $P_{mod}$ is identically zero) is well posed in the Gevrey class of order $1<s\leq 5$ (\cite{BN:jdam} see also \cite{Ni:book}). The Cauchy problem for $P_{mod}+SD_2$ with $0\neq S\in \C$ is not locally solvable in $\gamma^{(s)}$ for $s>3$ (Proposition \ref{pro:main} below) while the Cauchy problem for general second order operator $P$ of spectral type $2$ with ${\rm codim}\,\Sigma=3$ is well posed in the Gevrey class of order $1<s<3$  for any lower order term even if a tangent bicharacteristic exists (\cite{BN:JHDE}).

\section{A family of exact solutions}


First recall
\begin{lemma}
\label{lem:b2:zero}{\rm(\cite[Proposition 8.1]{Ni:book})}The Cauchy problem for $P_{mod}+\sum_{j=0}^1b_jD_j$ is not locally solvable in $\gamma^{(s)}$ at the origin  for any $b_0, b_1\in \C$ if $s>5$.
\end{lemma}
Thus in order to prove Theorem \ref{thm:main:bis} it suffices to prove the following result which also improves \cite[Theorem 1.3]{BN:JHDE}.
\begin{pro}
\label{pro:main}The Cauchy problem for $P_{mod}+\sum_{j=0}^2b_jD_j$ is not locally solvable in $\gamma^{(s)}$ at the origin  for any $b_0, b_1\in \C$ and $0\neq b_2\in\C$ if $s>3$.
\end{pro}
To prove Proposition \ref{pro:main} we repeat the proof of \cite[Theorem 1.3]{BN:JHDE} with obvious minor changes. Look for a family of solutions to  $(P_{mod}+\sum_{j=0}^2b_jD_j)U_{\lambda}=0$ in the form
\begin{equation}
\label{eq:teigi:U}
U_{\lam}(x)=e^{ i\xi_0\lam x_0}V_{\lam}(x'),\quad V_{\lam}(x')=e^{\pm \lam^5x_2-i(b_1/2)x_1}u(\lam^2x_1),\;\;\xi_0=\xi_0(\lam)
\end{equation}
that is, we look for $u(x)$ satisfying
\begin{equation}
\label{eq:a8}
u''(x)=\big(x^3+ 2\xi_0 x+ b_2\lam-\xi_0^2\lam^{-2}+ b_0\xi_0\lam^{-3}-b_1^2\lam^{-4}/4\big)u(x).
\end{equation}
To study solutions to \eqref{eq:a8} we consider 
\begin{equation}
\label{eq:aa1}
u''(x)=(x^3+a_2x+a_3)u(x),\quad x\in\C,\;\;a_j\in\C,\;\;j=2, 3.
\end{equation}
Let ${\mathcal Y}_0(x;a)$, $a=(a_2,a_3)$ be the  solution given in \cite[Chapter 2]{Si} to (\ref{eq:aa1}) which has asymptotic representation
\begin{equation}
\label{eq:zenkin:a}
{\mathcal Y}_0(x; a)\simeq x^{-3/4}\big[1+\sum_{N=1}^{\infty}B_Nx^{-N/2}\big]e^{-E(x,a)}
\end{equation}
as $x$ tends to infinity in any closed subsector of the open sector $|\arg x|<3\pi/5$ where
\[
E(x,a)=\frac{2}{5}x^{5/2}+a_2x^{1/2}
\]
and $A_N$, $B_N$ are polynomials in $(a_2,a_3)$. Let $\omega=e^{2\pi i/5}$ and set
\begin{equation}
\label{eq:setuzoku}
{\mathcal Y}_k(x;a)={\mathcal Y}_0(\omega^{-k}x; \omega^{-2k}a_2,\omega^{-3k}a_3),\quad k=0, 1, 2, 3, 4
\end{equation}
which are also solutions to (\ref{eq:aa1}).  Recall that \cite[Chpater 17]{Si}
\[
{\mathcal Y}_k(x;a)=C_k(a){\mathcal Y}_{k+1}(x;a)-\om {\mathcal Y}_{k+2}(x;a)
\]
where $C_k(a)$ are entire analytic in $a=(a_2,a_3)$ and $C_k(a_2,a_3)=C_0(\omega^{-2k}a_2,\omega^{-3k}a_3)$. Choose 
\begin{equation}
\label{eq:aa2}
\begin{split}
&u(x)={\mathcal Y}_0(x; a)=C_0(a){\mathcal Y_1}(x; a)-\omega{\mathcal Y_2}(x; a),\\
 &a_2=2\xi_0,\quad a_3= b_2\lam-\xi_0^2\lam^{-2}+ b_0\xi_0\lam^{-3}-b_1^2\lam^{-4}/4
\end{split}
\end{equation}
which solves \eqref{eq:a8}. We require that  $V_{\lam}(x')$  is bounded  as  $\lam\to +\infty$ when $|x'|$ remains in a bounded set. Note  \eqref{eq:setuzoku} and
\begin{equation}
\label{eq:heiho}
\begin{split}
(\om^{-1}x)^{5/2}=-i|x|^{5/2},\quad 2\om^{-2} \xi_0(\om^{-1}x)^{1/2}=-2i \xi_0|x|^{1/2},\\
(\om^{-2}x)^{5/2}=i|x|^{5/2},\quad 2\om \xi_0(\om^{-2}x)^{1/2}=2i \xi_0|x|^{1/2}
\end{split}
\end{equation}
for $x<0$. Then $|{\mathcal Y}_1(\lam^2x; a_2, a_3)|$ and $|{\mathcal Y}_2(\lam^2x;a_2, a_3)|$ behaves like $e^{2{\mathsf{Re}}(i\xi_0)\lam |x|^{1/2}}$ and $e^{-2{\mathsf{Re}}(i\xi_0)\lam |x|^{1/2}}$  respectively as $x\to -\infty$.  Since $\omega\neq 0$, taking \eqref{eq:aa2} into account, the requirement for boundedness implies that
\begin{align}
\label{eq:Stokes}
C_0(2\xi_0, a_3)=0,\\
\label{eq:expo}
-{\mathsf{Re}}(i\xi_0)={\mathsf{Im}}\,\xi_0<0.
\end{align}

Instead of solving directly the ``Stokes equation" (\ref{eq:Stokes}) we go rather indirectly. Let us consider
\[
H(\be)=p^2+x^2+i\be x^3
\]
as an operator in $L^2({\mathbb R})$ with the domain $D(H(\be))=D(p^2)\cap D(x^3)$. Here $p^2$ denotes the self-adjoint realization of $-d^2/dx^2$ defined in $H^2({\mathbb R})$ and by $D(x^3)$ we mean the domain of the maximal multiplication operator by the function $x^3$.

\begin{pro}
\label{pro:Caliceti_1}{\rm \cite[Corollary 2.16, Lemma 3.1]{CGM})}
 Let $k\in {\mathbb N}_0$ and $\ep>0$ be given. Then there is a $B>0$ such that for $|\be|<B$, ${\mathsf {Re}}\,\be>0$, $H(\be)$ has exactly one eigenvalue $E_k(\be)$ near $2k+1$. Such eigenvalues are analytic functions of $\be$ for $|\be|<B$, ${\mathsf{Re}}\,\be>0$, and admit an analytic continuation across the imaginary axis to the whole sector $|\be|<B$, $|\arg{\be}|<\frac{5\pi}{8}-\ep$ and 
 uniformly asymptotic to the following  formal Taylor expansion in powers of $\be^2$
\begin{equation}
\label{eq:Eigenvalue}
\sum_{j=0}^{\infty}a_{2j}\be^{2j},\quad a_0=2k+1
\end{equation}
 near $\be=0$ in any closed subsector in $|\arg{\be}|<\frac{5\pi}{8}-\ep$.
\end{pro}
For the proof we refer to \cite{CGM}, \cite{Simon}. Now we have
\begin{lemma}
\label{lem:a1}{\rm(\cite[Lemma 6.1]{BN:JHDE})}
Assume that ${\mathsf{Re}}\,\be>0$ and $E(\be)$ is an eigenvalue of the problem
\begin{equation}
\label{eq:a1}
-u''(x)+( x^2+i\be x^3)u(x)=E(\be)u(x)
\end{equation}
that is $(\ref{eq:a1})$ has a solution $0\neq u\in D(H(\be))$. Then we have
\begin{equation}
\label{eq:a2}
C_0\Big(-\frac{\om^2}{3}\be^{-\frac{8}{5}},\, \om^3\big\{\frac{2}{27}\be^{-\frac{12}{5}}+\be^{-\frac{2}{5}}E(\be)\bigl\}\Big)=0
\end{equation}
where $\be^{-j/5}=(\be^{-1/5})^{j}$ and the branch $\be^{\pm1/5}$ is chosen such that $|\arg \be^{\pm1/5}|<\pi/10$.
\end{lemma}
Let $E(\be)$, $|\be|<B$, ${\mathsf Re}\,\be>0$ be an eigenvalue which is analytically continued to the sector $|\be|<B$, $|{\rm arg}\,\be|<5\pi/8$ by Proposition \ref{pro:Caliceti_1} (though when $|{\rm arg}\,\be|=\pi/2$, $H(\be)$ admits  infinitely many distinct self-adjoint extensions, see \cite{CGM}). Since $C_0(a_2, a_3)$ is entire analytic in $(a_2, a_3)$ then \eqref{eq:a2} holds in this sector. Thanks to Lemma \ref{lem:a1}, if $\xi_0$ satisfies
\begin{equation}
\label{eq:a9}
\left\{\begin{array}{l}
\displaystyle{2\xi_0=-\frac{\omega^2}{3}\be^{-\frac{8}{5}},}\\[8pt]
\displaystyle{a_3=\omega^3\big\{\frac{2}{27}\be^{-\frac{12}{5}}+\be^{-\frac{2}{5}}E(\be)\big\}}
\end{array}\right.
\end{equation}
then \eqref{eq:Stokes} is satisfied hence we have
\[
u(\lam^2 x)={\mathcal Y}_0(\lam^2x;2\xi_0,a_3)
=-\om{\mathcal Y}_2(\lam^2x;2\xi_0, a_3).
\]
Therefore  we look for $\xi_0$ satisfying \eqref{eq:a9} and   \eqref{eq:expo}. Plugging $\xi_0=-\omega^2\be^{-8/5}/6$ into the second equation, (\ref{eq:a9}) is reduced to
\begin{equation}
\label{eq:10}
\begin{split}
\frac{2}{27}\om^3\be^{-\frac{12}{5}}+E(\be)
\om^3\be^{-\frac{2}{5}}+\frac{1}{36}
\om^4\be^{-\frac{16}{5}}\lam^{-2}\\
+ \frac{b_0}{6}\om^2\be^{-8/5}\lam^{-3}= b_2\lam-\frac{b_1^2}{4}\lam^{-4}.
\end{split}
\end{equation}
%

\subsection{Solving the equation (\ref{eq:10}) }
\label{sec:solv}

Solve (\ref{eq:10}) with respect to $\be$ where $E(\be)$ is one of $E_k(\be)$ which is analytic in $|\arg{\be}|<5\pi/8-\ep$ and admits a uniform asymptotic expansion (\ref{eq:Eigenvalue}) there by Proposition \ref{pro:Caliceti_1}. Put $\zeta=\be^{-2/5}=(\be^{-1/5})^{2}$ so that 
the equation leads to
\begin{equation}
\label{eq:reduced}
\om^3\zeta^6+\frac{27}{2}E(\zeta^{-5/2})\om^3\zeta
+\frac{3}{8}\om^4\zeta^8\lam^{-2}+ \frac{9}{4}b_0\om^2\zeta^4\lam^{-3}= \frac{27}{2}b_2\lam-\frac{27}{8}b_1^2\lam^{-4}
\end{equation}
and we look for a solution $\zeta(\lam)$ to \eqref{eq:reduced} verifying
\begin{equation}
\label{eq:cond}
|\arg \zeta(\lam)|<\pi/4-\ep,\quad {\mathsf{Im}}\;\om^2\zeta(\lam)^4>0
\end{equation}
where the second requirement comes from \eqref{eq:expo}. Assume that 
\begin{equation}
\label{eq:argA}
0<\arg b_2<\pi \quad \text{or}\quad -\pi \leq \arg b_2<-\pi/2
\end{equation}
and denote $27b_2 /2$ by  $A$  for notational simplicity and look for $\zeta(\lam)$ in the form
\[
\left\{\begin{array}{l}
\zeta(\lam)=e^{-\pi i/5}A^{1/6}\big(1+\lam^{-5/6}z\big)\lam^{1/6}
\quad (0<\arg A<\pi)
,\\[5pt]
\zeta(\lam)=e^{2\pi i/15}A^{1/6}\big(1+\lam^{-5/6}z\big)\lam^{1/6}\quad 
(-\pi\leq \arg A<-\pi/2).\end{array}\right.
\]
 It is clear that $\zeta(\lam)$ verifies (\ref{eq:cond}) provided that $z$ is bounded and $\lam$ is large. Note that   $E(\zeta^{-5/2})$ is analytic in $|\arg \zeta|<\pi/4-\ep$ for large $|\zeta|$ and verifies
\[
\big|E(\zeta^{
-5/2})-a_0\big|\leq C|\zeta|^{-5}
\]
with some $a_0=2k+1$, $k\in\N$ uniformly in $|\arg \zeta|<\pi/4-\ep$ when $|\zeta|\to \infty$. We insert $\zeta$ into (\ref{eq:reduced}) to get
\begin{equation}
\label{eq:A:la}
\begin{split}
A\lam(1+\lam^{-5/6}z)^6+\lam^{1/6}H(z)+d_1\lam^{-2/3}(1+\lam^{-5/6}z)^8\\
+d_2\lam^{-7/3}(1+\lam^{-5/6}z)^4=A\lam+d_3\lam^{-4},\quad d_i\in\C
\end{split}
\end{equation}
where $H(z)$ is analytic  in $|z|<B$ for $\lam\geq R$ and 
\[
|H(z)-a_0|\leq C\lam^{-5/6},\quad \lam\geq R.
\]
Note that $B$ can be chosen to be arbitrarily large taking $R$ large. Thus one can write \eqref{eq:A:la} as
\begin{equation}
\label{eq:zed}
6Az+a_0+\lam^{-5/6} F(z,\lam)=0
\end{equation}
where $F(z,\lam)$ is analytic in $|z|<B$ and bounded uniformly in  $|\lam|\geq R$. We may assume that $|a_0/6A|<B/2$ taking $R$ large as noted above. By Rouch\'e's thoerem we conclude that the equation (\ref{eq:zed}) has a solution $z(\lam)$ with $|z|<B$ for any $|\lam|\geq R_1$. Returning to $\be$ ($\zeta=\be^{-2/5}$) we conclude that (\ref{eq:10}) has 
 a solution of the form
\[
\left\{\begin{array}{l}
\beta(\lam)=iA^{-5/12}\lam^{-5/12}\big(1+\lam^{-5/6}z(\lam)\big)
\quad (0<\arg A<\pi),\\[5pt]
\beta(\lam)=e^{-\pi i/3}A^{-5/12}\lam^{-5/12}\big(1+\lam^{-5/6}z(\lam)\big) \quad 
(-\pi\leq \arg A<-\pi/2)\end{array}\right.
\]
where $|z(\lam)|<B$ for $\lam\geq R_1$. 
Plugging this $\beta(\lam)$ into \eqref{eq:a9} we get
\begin{pro}
\label{pro:sol} Assume \eqref{eq:argA}. Then there exists $\xi_0=\xi_0(\lam)$ such that
\[
\left\{\begin{array}{l}
2\xi_0=c\,\lam^{2/3}(1+\lam^{-5/6}z(\lam)),\quad {\mathsf{Im}}\;c<0,\\[4pt]
C_0(2\xi_0,\,  b_2\lam-\xi_0^2\lam^{-2}+ b_0\xi_0\lam^{-3}-b_1^2\lam^{-4}/4)=0
\end{array}\right.
\]
where $|z(\lam)|<B$ for $\lam>R$.
\end{pro}
If we write 
\begin{equation}
\label{eq:a:b:teigi}
2\xi_0=\lam^{2/3}a(\lam),\quad b_2\lam-\xi_0^2\lam^{-2}+ b_0\xi_0\lam^{-3}-b_1^2\lam^{-4}/4=\lam b(\lam)
\end{equation}
 it is clear that $|a(\lam)|$,  $|b(\lam)|<M$ with some $M>0$ for $\lam\geq R_1$.

\subsection{Asymptotics of ${\mathcal Y}_0(x;a\lam^{2/3},b\lam )$ for large $|x|$ and $\lam$ }

Recalling \eqref{eq:a:b:teigi} we study  how ${\mathcal Y}_0(x;a\lam^{2/3},b\lam )$ behaves for large $|x|$ and large $\lam$ where $|a|, |b| \leq M$ is assumed. 
In what follows $f=o_{a}(1)$ means that there are positive constants $C_{a}>0$ and $\delta_a>0$ such that
\[
|f|\leq C_a\lam^{-\delta_a}, \quad \lam\to+\infty.
\]
We make the asymptotic representation \eqref{eq:zenkin:a} slightly precise. 
\begin{pro}
\label{pro:Esti:Y}
Let $\rho>1/3$ be given. Then one can write
\begin{gather*}
{\mathcal Y}_0(x;a\lam^{2/3},b\lambda)
\simeq (1+p_{\rho}(x,\lambda))e^{R_{\rho}(x,\lambda)}x^{-3/4}\exp{\{-E_{\rho}(x; a, b, \lam)\}},\\
{\mathcal Y}_0'(x;a\lam^{2/3},b\lambda)
\simeq (-1+p_{\rho}(x,\lambda))e^{R_{\rho}(x,\lambda)}x^{3/4}\exp{\{-E_{\rho}(x; a, b,  \lam)\}}
\end{gather*}
as $x$ tends to infinity in any closed subsector of
\[
S_{\lam}=\{x;|\arg x|<3\pi/5, |x|> \lambda^{\rho}\}
\]
where  
\[
E_{\rho}(x; a, b, \lam)=\frac{2}{5}x^{5/2}+a\lam^{2/3} x^{1/2}-b\lam x^{-1/2}+r_{\rho}(x,\lam)
\]
and $r_{\rho}$ is a polynomial in $x^{-1/2}$ such that
\[
|r_{\rho}(x,\lam)|=C\lam^{4/3-3\rho/2}(1+o_{\rho}(1)),\quad  x\in S_{\lam}
\]
and $p_{\rho}(x,\lambda)$, $R_{\rho}(x,\lambda)$ are holomorphic in $S_{\lam}$ and in any closed subsector of $S_{\lam}$
\begin{gather*}
|p_{\rho}(x,\lambda)|\leq C_{\rho}\lambda^{-2(\rho-1/3)},\quad |R_{\rho}(x,\lambda)|\leq C_{\rho}\lam^{-1},\quad x\in S_{\lam}
\end{gather*}
and $R_{\rho}(x,\lam)\to 0$, $p_{\rho}(x,\lambda)\to 0$  as $|x|\to\infty$, $x\in S_{\lam}$.
\end{pro}
\begin{proof} 
We follow Sibuya \cite[Chapter 2]{Si} and \cite[Proposition 2.3]{BN:JHDE} with needed modifications. 
\end{proof}
\begin{lemma}
\label{lem:1:ten} Assume that ${\mathcal Y}_0(x;a\lam^{2/3},b\lam)$ verifies
\begin{equation}
\label{eq:rela}
{\mathcal Y}_0(x;a\lam^{2/3},b\lam)=-\om {\mathcal Y}_2(x;a\lam^{2/3},b\lam)
\end{equation}
and 
\begin{equation}
\label{eq:Im:a}
 {\mathsf{Im}}\,a= -\delta (1+o_a(1))
 \end{equation}
with some $\delta>0$. Let $X>0$. There exist $\ell$, $c>0$, $C>0$ such that for any $0\leq \mu<5/6$ we have
\begin{gather*}
\big|(d/dx)^k{\mathcal Y}_0(\lam^{2} x,\; a\lam^{2/3},b\lam)\big|\leq C(1+o_{\mu}(1))\lam^{\ell}\\
\times\exp{\big\{-\delta \lam^{5/3}|x|^{1/2}(1+o_{a}(1))+C_{\mu}(\lam^{\mu}+\lam^{-5/3+3\mu})\big\}}
\end{gather*}
for $k=0, 1$ and $|x|\geq \lam^{-2\mu}X$.
\end{lemma}
\begin{proof} In Proposition \ref{pro:Esti:Y} choose $\rho=2(1-\mu)$ and estimate ${\mathcal Y_0}$ in $x\leq -\lam^{-2\mu}X$ first. Recall that for $x<0$ we have
\begin{gather*}
{\mathcal Y}_0(x; a\lam^{2/3},b\lam)=-\omega{\mathcal Y_2}(x; a\lam^{2/3}, b\lam)
=-\om {\mathcal Y}_0(e^{\pi i/5}|x|;\om a\lam^{2/3},\om^{-1}b\lam).
\end{gather*}
Denote
\[
\phi^{-}(x,\lam)=E_{\rho}(e^{\pi i/5}\lam^2|x|; \omega a, \omega^{-1}b, \lam),\quad\phi^{+}(x,\lam)=E_{\rho}(\lam^2 x; \omega a, \omega^{-1}b, \lam)
\]
then we have
\begin{equation}
\label{eq:phi:x:fu}
\begin{split}
\phi^{-}(x,\lam)=\frac{2}{5}i\lam^5 |x|^{5/2}+i a\lam^{5/3}|x|^{1/2}
+ib|x|^{-1/2}+r_{\rho}(e^{\pi i/5}\lam^{2}|x|,\lam).
\end{split}
\end{equation}
Since $\big|r_{\rho}(e^{\pi i/5}\lam^{2}|x|,\lam)\big|\leq C_{\rho}\lam^{-5/3+3\mu}$ for $|x|\geq \lam^{-2\mu}X$ and then
\begin{equation}
\label{eq:x:fu}
-{\mathsf{Re}}\,\phi^{-}(x,\lam)
\leq -\delta\lam^{5/3}|x|^{1/2}(1+o_{a}(1))+C_{\rho}(\lam^{\mu}+\lam^{-5/3+3\mu}).
\end{equation}
  For $x\geq \lam^{-2\mu}X>0$ note that
\begin{equation}
\label{eq:x:sei}
\begin{split}
-{\mathsf{Re}}\;\phi^{+}(x,\lam)\leq -c\lam^5 x^{5/2}+C\lam^{5/3}x^{1/2}
+C_{\rho}(\lam^{\mu}+\lam^{-5/3+3\mu})\\
=-c\lam^{5-4\mu} x^{1/2}(1+o_{a}(1))+C_{\rho}(\lam^{\mu}+\lam^{-5/3+3\mu})
\end{split}
\end{equation}
for $\mu<5/6$. Then the assertion follows from \eqref{eq:x:fu} and \eqref{eq:x:sei}. 
\end{proof}
%
\begin{lemma}
\label{lem:X:1:2}Assume that \eqref{eq:rela} and \eqref{eq:Im:a} hold with some $\delta>0$ and that $\mu< 5/6$. Let $0<X_1<X_2<1$. Then there exist $C>0$, $\ell$ and $c>0$ such that 
\[
\big|{\mathcal Y}_0(\lam^{2} x; a\lam^{2/3}, b\lam)\big|\geq C\lam^{\ell} e^{-c\lam^{5/3-\mu}},\quad \lam^{-2\mu}X_1\leq -x\leq \lam^{-2\mu}X_2.
\]
\end{lemma}
\begin{proof}It is clear from \eqref{eq:phi:x:fu} that there exists $C_1>0$ such that
\[
-\lam^{5/3-\mu}/C_1\leq -{\mathsf{Re}}\,\phi^{-}(x, \lam)
\leq -C_1\lam^{5/3-\mu}
\]
when $\lam^{-2\mu}X_1\leq -x\leq \lam^{-2\mu}X_2$. This proves the assertion.
\end{proof}
\begin{lemma}
\label{lem:outside:many} Under the same assumptions as in Lemma \ref{lem:1:ten} there exist $c>0$, $A>0$ such that 
\begin{gather*}
\big|(d/dx)^k{\mathcal Y}_0(\lam^{2} x,\; a\lam^{2/3},b\lam)\big|\leq C_{\mu}A^{k+1}(1+k^{3}+\lam^{5/2})^{k}e^{c\lam^{5/6}},\;\;k\in {\mathbb N}
\end{gather*}
for  $|x|\geq \lam^{-2\mu}X$. 
\end{lemma}
\begin{proof} We first estimate $\big|(d/dx)^k{\mathcal Y}_0(\lam^{2}x;a\lam^{2/3},b)\big|$ in $x\leq -\lam^{-2\mu}X$. 
From Proposition \ref{pro:Esti:Y} with $\rho=2(1-\mu)$ we have
\[
{\mathcal Y}_0(\lam^{2}x;a\lam^{2/3},b\lam)=C(1+p_{\mu}(x))\lam^{-3/2}x^{-3/4}e^{-\phi^{-}(x, \lam)+R_{\mu}(x)}
\]
where $p_{\mu}(x)$ and $R_{\mu}(x)$ are holomorphic and bounded in $|x|>\lam^{-2\mu}X$, $|{\rm arg}\, x|<3\pi/5$. Since $|x|^{-1}\leq \sqrt{X}\lam^{2\mu}\leq \sqrt{X}\lam^{5/3}$ we have
\begin{gather*}
|d^k(-\phi^{-}(x, \lam)+R_{\mu}(x))/dx^k|\\
\leq C_{\mu}A^kk!(\lam^5 |x|^{3/2}+\lam^{5/3}|x|^{-1/2}+C_{\mu}|x|^{-3/2})|x|^{1-k}\\
\leq C_{\mu}A^kk!(1+\lam^5 |x|^{3/2}+\lam^{5/2})\lam^{5(k-1)/3}\\
\leq C_{\mu}A^kk!\lam^{5k/3}(\lam^{10/3}|x|^{3/2}+\lam^{5/6}),\quad |x|\geq \lam^{-2\mu}X,\;\;k\geq 1.
\end{gather*}
Therefore it follows that for $x\leq -\lam^{-2\mu}X$ 
\[
\big|d^ke^{-\phi^{-}(x, \lam)+R_{\mu}(x)}/dx^k|\leq CA^k\lam^{5k/3}(\lam^{10/3}|x|^{3/2}+\lam^{5/6}+k)^ke^{-{\mathsf{Re}}\,(-\phi(x)+R_{\mu}(x))}.
\]
Since $-{\mathsf{Re}}\,(-\phi^{-}(x, \lam)+R_{\mu}(x))\leq -c\lam^{5/3}|x|^{1/2}+C_{\mu}\lam^{5/6}$ for $ x\leq -\lam^{-2\mu }X$  
with $c>0$ independent of $\mu$ and 
\[
|x|^{3k/2}e^{-c\lam^{5/3}|x|^{1/2}}\leq C^{k+1}\lam^{-5k}k^{3k}
\]
we conclude that
\[
|d^ke^{-\phi^{-}(x, \lam)+R_{\mu}(x)}/dx^k|\leq CA^k(\lam^{5/2}+k^3)^ke^{c\lam^{5/6}}
\]
which proves the assertion for $x\leq -\lam^{-2\mu}X$. For $x\geq \lam^{-2\mu}X$ it is enough to repeat the same arguments noting that \eqref{eq:x:sei} and $5-4\mu>5/3$.
\end{proof}
%
\begin{cor}
\label{cor:outside} Assume that ${\mathcal Y}_0(x;a\lam^{2/3},b\lam)$ verifies \eqref{eq:rela} and \eqref{eq:Im:a}.  Then there exist $c>0$, $A>0$, $C>0$ such that for any $\ep>0$ there is $\lam_{\ep}$ such that
\[
\big|(d/dx)^k{\mathcal Y}_0(\lam^{2} x,\; a\lam^{2/3},b\lam)\big|\leq C_{\ep}A^{k+1}(1+k^{3}+\lam^{5/2})^{k}e^{c\lam^{5/6}},\;\;k\in {\mathbb N}
\]
for $\lam^{-5/3+\ep}\leq |x|$, $\lam\geq \lam_{\ep}$. 
\end{cor}
\begin{proof} Choose $\mu=5/6-\ep/2$ in Lemma \ref{lem:outside:many}.
\end{proof}

Next we estimate ${\mathcal Y}_0(\lam^{2}x;a\lam^{2/3},b\lam)$ for $|x|\leq \lam^{-5/3+\ep}$.

\begin{lemma}
\label{lem:time:T}{\rm (\cite[Lemma 6.5, Lemma 6.7]{BN:JHDE})}  Assume that $y(x,\lam)$ satisfies 
\begin{equation}
\label{eqn:1}
y''(x,\lambda)=(x^3+a\lam^{2/3} x+b\lambda)y(x,\lambda),\quad |a|,\;|b|\leq M.
\end{equation}
Then there are $c>0$, $C>0$ and $\ell_i>0$ such that for any $T>0$ we have
\begin{eqnarray*}
\big|(d/dx)^ky(x,\lambda)\big| \leq C^{k+1}(k+\lambda^{1/3}+|x|)^{3k/2}\lam^{\ell_1}(1+T)^{\ell_2
}e^{c\lam^{5/6}(1+\lam^{-1/3}T)^{5/2}}\\
\times \big\{|y(T,\lam)|+|y'(T,\lam)|\big\},\quad |x|\leq T,\;\;k\in\N,\;\;\lam\geq 1.
\end{eqnarray*}
\end{lemma}
%

%
\begin{pro}
\label{pro:inside:many} Assume that ${\mathcal Y}_0(x;a\lam^{2/3},b\lam)$ verifies \eqref{eq:rela}  and \eqref{eq:Im:a}.
Then there are $\ell$, $c>0$, $A>0$, $C>0$ such that for any $\ep>0$ there is $\lam_{\ep}$ such that
\begin{gather*}
\big|(d/dx)^k{\mathcal Y}_0(\lam^{2} x; a\lam^{2/3},b\lam)\big|\leq CA^{k+1}
 \lam^{\ell}
(k^{3}+\lam^{4})^{k}e^{c\lam^{5/6+\ep}},\;\;k\in\N
\end{gather*}
for $\lam\geq \lam_{\ep}$.
\end{pro}
 \begin{proof}
 Applying Lemma \ref{lem:1:ten} with $\mu=5/6-\ep/2$.  we have
 \[
  \big|(d/dx)^k{\mathcal Y}_0(\pm\lam^{1/3+\ep})\big|\leq C\lam^{\ell}e^{c_1\lam^{5/6}},\;\;\lam\geq \lam_{\ep},\;\; k=0, 1.
  \]
Since ${\mathcal Y}_0(x)={\mathcal Y}_0(x, a\lam^{2/3}, b\lam)$  satisfies \eqref{eqn:1}, choosing $T=\lam^{1/3+\ep}$ in Lemma \ref{lem:time:T} we get
\[
\big|(d/dx)^k{\mathcal Y}_0(x)\big|\leq C^{k+1}\lam^{\ell}(k+\lam^{1/3+\ep})^{3k/2}e^{c\lam^{5/6+3\ep}},\;\;|x|\leq \lam^{1/3+\ep},\;\;  k\in {\mathbb N}
\]
for $\lam\geq \lam_{\ep}$. This proves that
\begin{gather*}
\big|(d/dx)^k{\mathcal Y}_0(\lam^{2}x)\big|\leq C^{k+1}\lam^{\ell}(\lam^{2}k^{3/2}+\lam^{5/2+2\ep})^{k}e^{c\lam^{5/6+3\ep}}
\end{gather*}
for $|x|\leq \lam^{-5/3+\ep}$. Since $\lam^{k(5/2+2\ep)}e^{-\lam^{5/6+3\ep}}\leq k^{3k}$ and $
\lam^{2}k^{3/2}\leq C(\lam^{4}+k^3)$ 
combining Corollary \ref{cor:outside} with the above-obtained estimates we conclude the assertion.
\end{proof}
Recalling
\[
V_{\lam}(x')=e^{i\lam^5x_2-i(b_1/2)x_1}{\mathcal Y_0}(\lam^2 x_1; \lam^{2/3}a(\lam), \lam b(\lam))
\]
with $\lam^{2/3}a(\lam)=2\xi_0$ and $b(\lam)=b_2+b_0\xi_0\lam^{-4}-\xi_0^2\lam^{-3}-b_1^2\lam^{-5}/4$ one has
\begin{lemma}
\label{lem:V:bibun} There exist $c>0$, $A>0$, $C>0$  and $\lam_0>0$  such that 
\begin{equation}
\label{eq:V:bibun}
\big|\dif_{x_1}^{k}V_{\lam}(x')\big|\leq CA^{k}(k!)^3e^{c\lam^{4/3}},\quad k\in\N,\;\lam\geq \lam_0.
\end{equation}
\end{lemma}
\begin{proof}
Noting $\lam^{4k}\leq C_1^{k}(k!)^3e^{3\lam^{4/3}}$ the assertion follows from Proposition \ref{pro:inside:many}.
\end{proof}
%
 
\section{Proof of Theorem \ref{thm:main:bis}}

Assume that $b_2\neq 0$ satisfies \eqref{eq:argA}. Following Section \ref{sec:solv} we have a family of exact solutions $\{U_{\lambda}\}$ satisfying $(P_{mod}+\sum_{j=0}^2b_jD_j)U_{\lambda}=0$. We show that $\{U_{\lambda}\}$   
does not satisfy apriori estimates derived from $\gamma^{(s)}$ local solvability of the Cauchy problem if $s>3$. Let $h>0$ and a compact set $K$ be fixed and denote by $\gamma_0^{(s),h}(K)$ the set of all $f(x')\in \gamma^{(s)}(\R^2)$ such that ${\rm supp}\,f\subset K$ and \eqref{eq:gevrey} holds with some $C>0$ for all $\al\in \N^2$. Note that $\gamma_0^{(s), h}(K)$ is a Banach space with the norm
\[
\sup_{\al, x}\frac{|\dif_x^{\al}f(x')|}{h^{|\al|}|\al|!^s}.
\]
\begin{pro}[Holmgren]
\label{pro:Holmgren}Denote $
D_{\ep}=\{x\in \R^3\mid |x'|^2+|x_0|<\ep\}$. 
There exists $\ep_0>0$ such that for $\ep$ satisfying  $0<\ep<\ep_0$  if $u(x)\in C^2(D_{\ep})$ satisfies
\[
\begin{cases}
\big(P_{mod}+\sum_{j=0}^2b_jD_j\big)u=0\;\;\text{in}\;\; D_{\ep},\\\;D_0^ju(0, x')=0\;\;(j=0,1),\;\;x\in D_{\ep}\cap\{x_0=0\}
\end{cases}
\]
then $u(x)\equiv 0$ in $D_{\ep}$.
\end{pro}
\begin{lemma}
\label{lem:futosiki}{\rm(e.g.\cite[Proposition 4.1, Theorem 4.2]{Miz}, \cite{Lax})}
Assume that the Cauchy problem for $P_{mod}+\sum_{j=0}^2b_jD_j$ is locally solvable in $\gamma^{(s)}$ at the origin. Then there exists $\delta>0$ such that for any  $0<\ep_1<\delta$ and any $\Phi=(u_j(x'))\in \gamma_0^{(s),h}(\{|x'|\leq \ep_1\})$ there is a unique solution $u(x)\in C^2(D_{\delta})$ to the  Cauchy problem \eqref{eq:CP} with $U_{\Phi}=D_{\delta}$ and for any compact set $L\subset D_{\delta}$ there exists $C>0$ such that 
\begin{equation}
\label{eq:G:hyoka}
|u(x)|_{C^2(L)}\leq C\sum_{j=0}^1\sup_{\al, x'}\frac{|\dif_{x'}^{\al}u_j(x')|}{h^{|\al|}|\al|!^s}
\end{equation}
holds.
\end{lemma}
Since $\lam^{5k}\leq k!^se^{s\lam^{5/s}}$ and $U_{\lam}(0, x')=V_{\lam}(x')$, $D_0 U_{\lam}(0, x')= \xi_0(\lam)\lam V_{\lambda}(x')$ it is clear from Lemma \ref{lem:V:bibun} that one can find $c_1>0$, $C>0$ such that
\begin{equation}
\label{eq:U:futosiki}
\sum_{j=0}^1\sup_{\al, x'}\frac{|\dif_{x'}^{\al}D_0^jU_{\lam}(0, x')|}{h^{|\al|}|\al|^{s|\al|}}\leq Ce^{c_1\lam^{\max{\{5/s, 4/3\}}}},\quad s\geq 3.
\end{equation}
On the other hand thanks to Lemma \ref{lem:X:1:2} 
and Proposition \ref{pro:sol} there is $c_0>0$ such that
\begin{equation}
\label{eq:U:x:0}
\big|U_{\lam}( x_0, -\lam^{-2\mu}, 0)\big|\geq C\lam^{\ell}e^{c_0\lam^{5/3} x_0-c\lam^{5/3-2\mu}},\quad x_0>0
\end{equation}
where $\mu$ is chosen such that $0<\mu<5/6$. Let $\chi(x')\in \gamma^{(s)}(\R^2)$ be such that $\chi(x')=0$  for $|x'|\geq \sqrt{\ep_1}$ and $\chi(x')=1$ for $|x'|\leq \sqrt{\ep_2}<\sqrt{\ep_1}$. Since  $\Phi_{\lam}=\chi(x')(U_{\lam}(0, x'), D_0U_{\lam}(0, x'))\in \gamma_0^{(s), h}(\{|x'|\leq \ep_1\})$, thanks to Lemma \ref{lem:futosiki},  there is a unique solution $u_{\lam}(x)\in C^{2}(D_{\delta})$ to the Cauchy problem  with Cauchy data  $\Phi_{\lam}(x')$ which satisfies \eqref{eq:U:futosiki}. 
Thanks to Proposition \ref{pro:Holmgren} we see that $u_{\lam}=U_{\lam}$ in $D_{\ep_2}$. Take a compact set $L\subset D_{\ep_2}$ that contains $( x_0, -\lam^{-2\mu}, 0)$ with small $x_0>0$ and large $\lam$.  If $s>3$ hence  $\max{\{5/s, 4/3\}}<5/3$ the inequalities \eqref{eq:U:futosiki} and \eqref{eq:U:x:0} are not compatible  which proves Theorem \ref{thm:main:bis}.

When $b_2\neq 0$ does not satisfy \eqref{eq:argA} we make a change of local coordinates $(x_0, x_1, x_2)\mapsto (-x_0, x_1, -x_2)$ such that $P_{mod}+\sum_{j=0}^2b_jD_j$ will be
\begin{equation}
\label{eq:sin:eq}
P_{mod}-b_0D_0+b_1D_1-b_2D_2.
\end{equation}
in the new local coordinates. Since the local solvability in $\gamma^{(s)}$ at the origin is invariant under (analytic) change of local coordinates and $-b_2$ obviously satisfies \eqref{eq:argA}, we conclude the same assertion also in this case.

\end{document}